\documentclass[11pt]{amsart}
\usepackage{mathrsfs}
\usepackage{graphicx}
\usepackage{verbatim}
\usepackage{epsfig}
\usepackage{amssymb}
\usepackage{amsthm}
\usepackage[all,knot,arc]{xy}
\usepackage{color}
\usepackage{hyperref}
\hypersetup{colorlinks=true,citecolor=blue,linkcolor=blue}
\newtheorem{theorem}{Theorem}[section]
\newtheorem{lemma}[theorem]{Lemma}
\newtheorem{proposition}[theorem]{Proposition}

\newtheorem{question}[theorem]{Question}

 \usepackage{palatino}

\theoremstyle{definition}
\newtheorem{definition}[theorem]{Definition}

\newtheorem{remark}[theorem]{Remark}

\newtheorem{claim}[theorem]{Claim}

\numberwithin{equation}{section}

\begin{document}

\author{Amey Kaloti}
\address{Department of Mathematics \\ University of Iowa}
\email{amey-kaloti@uiowa.edu}

\author{B\"{u}lent Tosun}
\address{Department of Mathematics \\ University of Virginia}
\email{bt5t@virginia.edu}

\title[Hyperbolic Rational Homology Spheres]{Hyperbolic Rational Homology Spheres not Admitting Fillable Contact Structures.}

\begin{abstract}
In this short note, we exhibit an infinite family of hyperbolic rational homology $3$--spheres which do not admit any fillable contact structures. We also note that most of these manifolds do admit tight contact structures. 
\end{abstract}
\maketitle

\section{Introduction}

An important question in low dimensional contact topology is whether a given closed, oriented, irreducible $3$-manifold admits a tight contact structure or not. Recall that an overtwisted disc in a $3$--manifold $M$ is an embedded disc $D$ in $M$ such that the $\partial D$ is Legendrian and the Thurston-Bennequin number, denoted $tb$ of $\partial D$ is exactly $0$. A manifold is called \textit{overtwisted} if it contains overtwisted disc, otherwise it is called \textit{tight}. An advance in this direction was made by Eliashberg and Thurston~\cite{Eliashberg_Thurston_Confoliation} following work of Gabai~\cite{Gabai_Taut_Foliation}. We recall this briefly. It follows from work Gabai~\cite{Gabai_Taut_Foliation} that, if $M$ is an irreducible $3$--manifold and $\Sigma$ is an oriented surface realizing a non-trivial homology class in $M$ and of minimal genus among representatives of its homology class, then there is a taut foliation $\xi$ of $M$ with $\Sigma$ as a leaf. Moreover, Gabai can construct these taut foliations on any irreducible $3$--manifold with $b_1(M) > 0$. Now it follows from work of Eliashberg-Thurston~\cite{Eliashberg_Thurston_Confoliation} that, with $M$ and $\xi$ as above, there is a fillable and hence tight (see discussion below) contact structure $\xi'$. Hence, we know that any irreducible $3$ manifold with nontrivial second homology admits a tight contact structure. Work of Honda-Kazez-Matic~\cite{HKM} (and independently of Colin \cite{Colin99}) provides a tight contact structure on toroidal $3$--manifolds (regardless of homologically essential condition). Also, it follows from the work of Lisca and Stipsciz~\cite{Lisca_Stipsciz_Classification_of_manifolds_admitting_tight_contact_Structures} that, a {\it small} Seifert fibered $3$--manifold admits a tight contact structure if and only if it is not a $2q-1$ surgery on $(2,2q+1)$ torus knot, for $q \geq 1$.

So this leaves us with the following open question.
\begin{question}
\label{hyperbolic_question}
Does every hyperbolic $3$--manifold which is a rational homology sphere admit a tight contact structure?
\end{question}

A related notion to tightness is that of a symplectic filling whose definition we now recall.
\begin{definition}
A contact manifold $(M,\xi)$ is \textit{weakly fillable} if $M$ is the oriented boundary of a symplectic manifold $(X,\omega)$ and $\omega|_{\xi} > 0$. 
\end{definition}

We remark here that there are notions of strong and Stein fillability which are known to be strictly stronger than the weak fillability that we have defined here (See~\cite{Eliashberg_Torus_Examples, Ghiggini_strong_not_Stein}). For the rest of the paper fillable will stand for weakly fillable.

It follows from a theorem of Eliashberg and Gromov that if a $(M,\xi)$ is fillable then it is tight. It is also  known that fillability is strictly stronger than tightness, see~\cite{Etnyre_Honda_nonfillable}. So we can refine the Question~\ref{hyperbolic_question} and ask: Does every hyperbolic $3$--manifold which is a rational homology sphere admit a fillable contact structure? We answer this question negatively.

\begin{theorem}
\label{main_theorem}

Let $M_{r,m}$ be a $3$--manifold obtained by performing a rational $r$--surgery along $P(-2,3,2m+1)$, $m\geq 3$ pretzel knot (see 
Figure~\ref{figure_pretzel_knot}), then

\begin{enumerate}
\item $M_{r,3}$ does not admit fillable contact structure for any $r \in [9,15]$.

\item For $m$ sufficiently large, there is a subinterval $[2m+3,s]\subset[2m+3,2m+\frac{\sqrt{16m+13}+9)}{2})$ such that $M_{r,m}$ does not admit fillable contact structure for any $r \in [2m+3,s]$. 

\end{enumerate}

Moreover, all of these manifolds do admit a tight contact structure for every $r\neq 2m+3$.
\end{theorem}

It follows from the classification of exceptional surgeries on these knots that they are all hyperbolic $3$--manifolds. See~\cite{Meier}, ~\cite{Wu2011}, ~\cite{IJ2009} and ~\cite{FIKMS2009}.

\begin{figure}[h!]
\begin{center}
  \includegraphics[width=8cm]{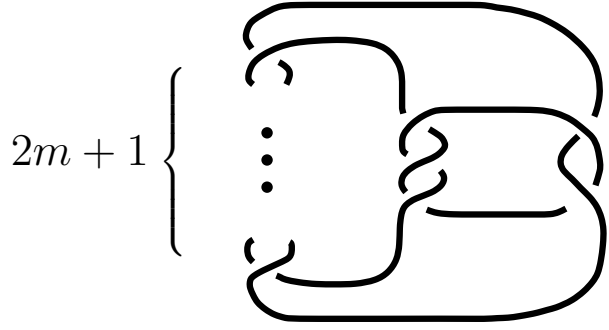}
 \caption{ The pretzel knot $P(-2,3,2m+1)$.}
  \label{figure_pretzel_knot}
\end{center}
\end{figure}

We want to make a couple of remarks about part~$(2)$ of Theorem~\ref{main_theorem}.

\begin{itemize}

\item  In fact {\it for any $m$} as soon as one can prove that there is a squarefree integer $s$ in the interval $I_m=[2m+3,2m+\frac{\sqrt{16m+13}+9)}{2})$, then we will be able to show that $M_{r,m}$ does not admit fillable contact structure for any $r \in [2m+3,s]$. For example for $m=4$, the interval $I_4=[11,17)$ contains three squarefree integers, taking the largest one $s=15$, our argument in Section~\ref{fill} will show that $M_{r,4}$ does not admit fillable contact structure for any $r \in [11,15]$. This seems to work for any $m$ but in general it is a very delicate number theory problem to guarantee the existence of a squarefree integer in a given interval. As it will be explained below for $m$ sufficiently large, there is always at least one squarefree integer in the interval $(2m+3,2m+\frac{\sqrt{16m+13}+9)}{2})$.         

\item The intervals of slopes mentioned in Theorem~\ref{main_theorem} are not exact intervals. In fact we expect that for any surgery slope in the interval $[2m+3,4m+6)$, the manifolds $M_{r,m}$ do not admit any fillable contact structures whenever $m>3$ (for $m=3$, due to surprising smaller exceptional surgeries on $P(-2,3,7)$, interval we expect is $[9,17)$). See Section~\ref{fill} for more details.

\item Our goal in this note, motivated by Question~\ref{hyperbolic_question}, was rather to exhibit existence of examples of rational homology spheres which are hyperbolic manifolds that do not have fillings and point out a curious subfamily of those non-fillable manifolds, see Question~\ref{tight?}, on which existence of a tight contact structure remains unclear.   
\end{itemize}

\bigskip

We also want to mention that Etnyre-Baldwin in \cite{EB2013} proved that there are hyperbolic manifolds with tight non-fillable contact structures. It is however not known whether these manifolds admit any other fillable structures.  

\bigskip

It was pointed out to us by John Etnyre that Youlin Li and Yajing Liu in \cite{Liliu} independently proved similar results.   

\smallskip
\textit {Acknowledgement.}
We are  grateful to John Etnyre for his support and interest in our work. We would also like to thank Tom Mark for valuable comments and discussions, and Michael Filaseta for the reference \cite{FGT2015}. The second author was supported in part by AMS-Simons travel grant.

\section{Preliminaries}
In this section we want to briefly recall some definitions and facts mainly from Heegard Floer theory related to our knots $P(-2,3,2m+1)$ and their surgeries.

\subsection{L-space surgeries} Recall that a rational homology $3$ sphere $Y$ is called an $L$--space if $|H_1(Y;\mathbb{Z})|  = rk(\widehat{HF(Y)})$. Here $\widehat{HF(Y)}$ is the hat version of Heegaard Floer homology group of $Y$. 

A knot $L$ in $Y$ is called a $L$--space knot if a large enough positive surgery on $L$ produces an $L$-space. It follows from work of Fintushel-Stern~\cite{Fintushel_Stern_Pretzel_knots}, Bleiler and Hodgson~\cite{Bleiler_Hodgson}, and Oszvath and Szabo~\cite{Oszvath_Szabo_Lens_Space_Surgeries} that $P(-2,3,2m+1)$ pretzel knots as in our Theorem~\ref{main_theorem} are $L$--space knots. Moreover, for these knots it is known that the slice genus denoted $g_s$ is $m+2>0$ and a result of Lisca-Stipsicz in ~\cite[Proposition $4.1$]{Lisca_Stipsicz_Oszvath_Szabo_invariants_1} gives that the smooth $3$--three manifold $M_{r,m}$ obtained by performing a rational $r$ surgery along pretzel knot $P(-2,3,2m+1)$ is an $L$--space for every $r\geq 2g_s-1=2m+3.$    

\subsection{Negative definite fillings.} The following theorem, due to Ozsv\'{a}th and Szab\'{o}, provides an obstruction to the existence of a symplectic filling for $L$-spaces.

\begin{theorem}[Theorem~$1.4$,\cite{Ozsvath_Szabo_L_Space}]\label{osfilling}
If $Y$ is an $L$--space, then all of its symplectic fillings are negative definite.
\end{theorem}

The first of the following theorems of Owens and Strle provides an obstruction for a certain $3$--manifold to bound a negative definite manifold. The second shows negative definite fillings of a $3$--manifold that is obtained by positive rational surgery do not have any ``gap".  
  
\begin{theorem}[Theorem 2, \cite{Owens_Strle_Characterization_of_lattices}]
\label{theorem_owens_strle}
Let $Y$ be a rational homology $3$--sphere with $|H_1(Y;\mathbb{Z})| = \delta$. If $Y$ bounds a negative definite four manifold $X$, and if either $\delta$ is square free or there is no torsion in $H_1(X;\mathbb{Z})$, then

\begin{equation}
 \max_{ \mathbf{t} \in Spin^c(Y)} 4d(Y, \mathbf{t}) \geq 
  \begin{cases} 
   1-\frac{1}{\delta} & ~\text{for } \delta ~\text{odd} \\
    1   & \text{for } ~\delta ~\text{even}
  \end{cases}
\end{equation}
\end{theorem}

\begin{theorem}[Theorem~1(b), ~\cite{Owens_Strle_Negative_Definite_four_manifolds}]
\label{theorem_owens_strleNew}
For all rational numbers $r > m(K)$, the $r$-surgery $S^3_r(K)$ on $K$ bounds a negative definite
manifold $X_r$, with $H^2(X_r)\longrightarrow H^2(S^3_r(K))$ surjective where 
\[
m(K)=\inf\{r\in\mathbb{Q}_{\geq0}|S^3_r(K)~~ \textrm{bounds~~a~~negative-definite~~4~-~manifold}\}.
\]
\end{theorem}

\section{Proof of Theorem~\ref{main_theorem}}

In this section we give our proof of Theorem~\ref{main_theorem}.

\subsection{Tightness.} Recall $M_{r,m}$ denotes a smooth rational $r$--surgery along the pretzel knot $P(-2,3,2m+1)$. We prove that for any $r \neq 2m+3$, $M_{r,m}$ does admit a tight contact structure. To end this, we first note that by work of Ng in~\cite [Theorem $6.1.4$]{Ng2001_thesis} the pretzel knot $P(-2,3,2m+1)$ has the maximal Thurston-Bennequin number  $\overline{tb}(P(-2,3,2m+1))=2m+3$. In particular, since the genus of $P(-2,3,2m+1)$ is $m+2$, we obtain that the knot type $P(-2,3,2m+1)$ realizes the Bennequin bound. If the surgery coefficient $r<\overline{tb}=2m+3$, then, for each $m$, such surgery can be obtained by a contact surgery on a Legendrian link for which all contact surgery coefficients are negative, in particular it describes a Stein fillable contact structure, and hence it is tight. When $r>2m+3$, smooth $r$--surgery can also be obtained by a suitable contact surgery on a Legendrian link. It follows from work of Lisca-Stipsicz in~\cite [Theorem $1.1$]{Lisca_Stipsicz_Oszvath_Szabo_invariants_1} that   
this contact surgery diagram describes a contact structure with non-zero Heegaard
Floer contact invariant, and hence it is tight. For slope $r=2m+3$, there is an (obvious) corresponding contact surgery with contact framing zero, and such a contact surgery, by definition, results an overtwisted contact structure. One then naturally wonders about the following question.

\begin{question}\label{tight?}
Does $M_{2m+3,m}$ also support a tight contact structure?
\end{question}
 
\begin{remark}
We want to note the pretzel knots $P(-2, 3, 1), P(-2, 3, 3)$ and $P(-2, 3, 5)$ are isotopic to the positive torus knots $T(2, 5)$, $T(3, 4)$, and $T(3, 5)$, respectively. Owens and Strle in~\cite{Owens_Strle_Negative_Definite_four_manifolds} ({\it cf.}~\cite[Theorem~$1.3$]{LecuonaLisca}), extending results from~\cite[Theorem~$4.2$]{Lisca_Stipsicz_Oszvath_Szabo_invariants_1} found exact interval of surgeries on positive torus knots for which the surgery manifold does not admit fillable contact structures. On the other hand, there is exactly one surgery coefficient, $r=3$ and only one of these knot, the torus knot $T(2,5)$ (or $P(-2,3,1)$) for which the resulting surgery is known~\cite{Lisca_Stipsciz_Classification_of_manifolds_admitting_tight_contact_Structures} to yield a manifold without any positive tight contact structures. 
\end{remark}

\subsection{Fillability.}\label{fill} In this section we study fillings of $M_{r,m}$. Our goal is to show that for any $P(-2,3,2m+1)$ pretzel knot we can calculate a range of slopes, such that the surgeries with these slopes do not bound a negative definite manifold. Indeed we will try to find largest possible interval such that for each integer in the interval, the inequality in Theorem~\ref{theorem_owens_strle} fails to hold.

We want to note $M_{4m+6,m}$ is orientation preserving diffeomorphic to the small Seifert fibered space $M(-2, \frac{1}{2},\frac{1}{4},\frac{2m-7}{2m-5})$  which obviously is a boundary of a negative definite manifold for each $m\geq 3$, and all tight contact structures on $M(-2, \frac{1}{2},\frac{1}{4},\frac{2m-7}{2m-5})$ are Stein fillable (See \cite{Ghiggini2008}). Moreover as explained above that for any $r<2m+3$, $M_{r,m}$ do admit a Stein fillable contact structure. So necessarily the interval of surgeries that result non-fillable contact manifolds must be a subset of $[2m+3,4m+6)$ for $m\geq 4$ ( and $[9,17)$ for $m=3$, due to surprising smaller exceptional surgeries on $P(-2,3,7)$.)

\subsubsection{Fillability of $M_{r,3}$.} Let $K$ denote the $P(-2,3,7)$ pretzel knot from now on. Its slice genus $g_s(K) = 5$. It follows from Proposition~$4.1$ of~\cite{Lisca_Stipsicz_Oszvath_Szabo_invariants_1} that the manifold obtained by any rational surgery $r \geq 9$ on $K$ is an $L$--space. So we show that $M_{r,3}$ cannot bound a negative definite $4$--manifold for any $r\in(0,15]$. Towards that end we use Theorem~\ref{theorem_owens_strle} and~ \ref{theorem_owens_strleNew} of Owens and Strle.

We compute the $d$ invariant for the manifold obtained as $15$--surgery on $P(-2,3,7)$ pretzel knot. For this we use the formula for $d$ invariants given as Theorem~$6.1$ in~\cite{Owens_Strle_Characterization_of_lattices} which we recall. If $L \subset S^3$ is an $L$--space knot and $t_i(L)$ denote its torsion coefficients. Then for $n > 0 $ the $d$-invariants of the $\pm n$ surgery on $L$ are given by 
\[d(L_n,i) = d(U_n,i)-2t_i(L), d(L_{-n},i) = - d(U_n,i) \] 
\noindent
for 	$|i| \leq n/2$ being a $Spin^c$ structure on $L_n$, where $L_n$ denotes the manifold obtained by $n$ surgery along knot $L$ and $U_n$ denotes manifold obtained by $n$ surgery along the unknot. The $d$ invariants for $U_n$ are given as \[d(U_n,i) = \frac{(n-2|i|)^2}{4n} - \frac{1}{4}\]
\noindent
Recall that, if $L \subset S^3$ is an $L$--space knot and \[\Delta_L(T) = a_0 + \sum_{j>0} 	a_j (T^j+T^{-j}) \] its symmetrized Alexander polynomial, then the torsion coefficients of $L$ are given as
\[t_i(L)  = \sum_{j > 0 } j a_{|i|+j} \]. 

In case of $(-2,3,7)$ pretzel knot $K$, the Alexander polynomial is given as
\[ \Delta_K(T) = T^5 -T^4 + T^3 - T^2+ T -1 + T^{-1} - T^{-2}+T^{-3} -T^{-4}+T^{-5} \]

So the torsion coefficients in this case are given as \[t_0(K) = 3, t_1(K) = t_{-1}(K) = 2, t_2(K) = t_{-2}(K) = 2, t_3(K) = t_{-3}(K) =1, t_4(K) = t_{-4}(K) = 1,\] \[t_i(K) = 0 \mbox{  } for \mbox{  } |i| \geq 5 \]

So the $d$ invariants for $M_{15,3}$ are calculated as 
\[
 d(M_{15,3},i) = 
  \begin{cases} 
   -\frac{5}{2} & ~\text{if } i=0 \\
    -\frac{43}{30}   & \text{if } |i|=1 \\
    -\frac{67}{30}   & \text{if } |i|=2 \\
    -\frac{27}{30}   & \text{if } |i|=3 \\
    -\frac{43}{30}   & \text{if } |i|=4 \\
     ~~\frac{1}{6}     & \text{if } |i|=5 \\
    -\frac{1}{10}    & \text{if } |i|=6 \\
    -\frac{7}{30}    & \text{if } |i|=7 
  \end{cases}
\]

In particular for any $i \leq 7$,  $4 d(M_{15,3},i) < \frac{14}{15}$, that is the inequality in Theorem~\ref{theorem_owens_strle} fails to hold for squarefree surgery coefficient $r=15$. Hence $M_{15,3}$ cannot bound a negative definite $4$--manifold. Now it follows from Theorem~\ref{theorem_owens_strleNew} that for any rational $0<r \leq 15$, the surgered manifold $M_{r,3}$ cannot bound a negative definite $4$--manifold. Moreover, since $M_{r,3}$ is an $L$--space for any $r\geq9$, by Theorem~\ref{osfilling} we obtain that $M_{r,3}$ does not admit fillable structure for any $r\in[9,15]$.  

\subsubsection{Fillabality of $M_{r,m}$.} We return to the general case. The idea of proof is already provided above. We just need to check some computational details. In what follows, we provide computations of torsion coefficients $t_i(P(-2,3,2m+1))$ in Lemma~\ref{lemma_torsion_coefficient}, then in Lemma~\ref{boundlemma}, we provide an interval such that for each integer in the interval the corresponding surgery manifold fails to hold the inequality in Theorem~\ref{theorem_owens_strle}.

The symmetrized Alexander polynomial for $P(-2,3,2m+1)$ is given by
\[\Delta_K(T) = T^{m+2} - T^{m+1}+\dots - T^{-(m+1)} + T^{-(m+2)} \]
This can be written as 
\[\Delta_K(T) = a_0 + \sum_{j>0} a_j(T^j+T^{-j}) \]
\noindent
here $a_j = \pm 1$, for any $j$.

It can be seen easily from the formula for the Alexander polynomial that $a_i=(-1)^{i+m}.$

Following proposition will be used in the computation of the torsion invariants. Its proof is straightforward and is left for the reader as an exercise.

\begin{proposition}
\label{sum_proposition}
If $k > 1$ is an odd integer, then
\[1-2+3-4+ \dots + k = \frac{k+1}{2}.\]
If $k > 1$ is an even integer, then
\[-1+2-3+\dots+k = \frac{k}{2}. \] 

\end{proposition}

\begin{lemma}
\label{lemma_torsion_coefficient}
Torsion coefficients, $t_i(K)$, for the $(-2,3,2m+1)$ pretzel knot denoted $K$ are given as follows.
\begin{enumerate}
\item If $m$ is odd, then 

\[
 t_i(K) = 
  \begin{cases} 
   \frac{m+2-i}{2} & ~\text{if } i \leq m+2 ~\text{odd} \\
    \frac{m+3-i}{2}   & \text{if } i < m+2 ~\text{even} \\
    0 & i\geq m+2
  \end{cases}
\]

\item If $m$ is even, then

\[
 t_i(K) = 
  \begin{cases} 
   \frac{m+3-i}{2} & ~\text{if } i < m+2 ~\text{odd} \\
    \frac{m+2-i}{2}   & \text{if } i \leq m+2 ~\text{even} \\
    0 & i\geq m+2
  \end{cases}
\]

\end{enumerate}

\end{lemma}

\begin{proof}

We only prove the first part. The second part is exactly analogous. Assume that $m \geq 3$ and odd. Torsion coefficients are defined as $t_i(K) = \sum_{j>0} j a_{|i|+j} $. If $i$ is odd then 

\begin{align}
    t_{i}(K) &= a_{i+1} + 2 a_{i+2}+ \dots + (m+2-i) a_{m+2}  \\
      &= -1 +2 -3+ \dots + (m+2-i) 
\end{align}

It now follows from Proposition~\ref{sum_proposition} that in this case $t_i(K) = \frac{m+2-i}{2}$ as $(m+2-i)$ is even. Similarly, if $i$ is even, then
\begin{align}
    t_{i}(K) &= a_{i+1} + 2 a_{i+2}+ \dots + (m+2-i) a_{m+2}  \\
      &= 1 -2 +3- \dots + (m+2-i) 
\end{align}

Again by Proposition~\ref{sum_proposition} we get that $t_{i}(K) = \frac{m+3-i}{2}$.

\end{proof}

Now we compute the $d$-invariants for some surgeries. In particular, we show the following.

\begin{lemma}\label{boundlemma}
Let $K$ denote the $P(-2,3,2m+1)$ pretzel knot with $m \geq 3$. Let $M_{2k+m,m}$ denote the manifold obtained by $(2m+k)$ surgery along $K$ such that $k^2-9k<4m-17$. Then the inequality in Theorem~\ref{theorem_owens_strle} fails to hold.
\end{lemma}

We first make some more general observations about the $d$ invariants of the surgered manifold. 
\begin{claim}\label{claim1}
 The $d$ invariant $d(M,i)$ is negative for any $|i| < m+2$ and $k^2-9k<18m-4$.
\end{claim}

\begin{proof}
We compute the formula for $d$ invariant as
\begin{equation}\label{eqnn}
 d(M,i) = 
  \begin{cases} 
   \frac{(2m+k-2|i|)^2}{4(2m+k)} - \frac{1}{4} - (m+3-i) & ~\text{if}~ m+i ~\text{odd}~~\text{and}~ i < m+2 \\
    \frac{(2m+k-2|i|)^2}{4(2m+k)} - \frac{1}{4} - (m+2-i)   & ~\text{if}~ m+i ~\text{ even}~~\text{and}~ i < m+2
  \end{cases}
\end{equation}

To show that $d(M,i) < 0$ it suffices to show that the numerator of Equation~\ref{eqnn} for each choice of $m$ and $i$ is negative. This will not be true in general. So some conditions on $k$ are necessary. Assume $m$ is odd and $i<m+2$ is even. It is enough to check this for largest $i<m+2$ odd. That is for $i=m+1$. Simple calculation shows that the numerator is negative whenever $k^2-13k<18m-4$. The same calculation holds true when $m$ is even and $i<m+2$ is odd. Similarly one can show that when $m$ is odd (or even) and $i<m+2$ is odd (or even), then the numerator  $(2m+k-2|i|)^2-(2m+k)(1+4(m+3-i))$ is negative whenever $k^2-9k<18m$. So, by taking intersection of these we conclude that $d(M,i)$ is negative for any $|i| < m+2$ and $k^2-9k<18m-4$ 

\end{proof}

\begin{claim}
The $d$ invariant $d(M,i)$ is decreasing function for any $m+2 \leq|i|\leq 2m+2$. More importantly it assumes positive value for $i=m+2$.
\end{claim}
\begin{proof}
The proof is straightforward from the fact that the torsion coefficients $t_i(K)=0$ for any $i\geq g(K)=m+2$. 
\end{proof}
\begin{proof}[Proof of Lemma~\ref{boundlemma}]
Now combining claims above, we conclude that $d(M,i)$ attains its maximum value at $i=m+2$ and it is calculated as $d(M_{2m+k,m},m+2)=\frac{(2m+k-2m-4)^2-2m-k}{4(2m+k)}$. Now for any integer $N$ in the interval $[2m+3,2m+k]$ with $k^2-9k<4m-17$ (note such $k$'s satisfies $k^2-9k<18m-4$), it is easy to see that $4d(M_{2m+k,m},m+2)<\frac{2m+k-1}{2m+k}$ when $k$ is odd and $4d(M_{2m+k,m},m+2)<1$ when $k$ is even. The proof follows. 
\end{proof}

\begin{proof}[Proof of Theorem~\ref{main_theorem}]
 The fillability of $M_{r,3}$ and existence of tight contact structures on $M_{r,m}$ are already proved above. So, we are left to prove part $(2)$ of the Theorem~\ref{main_theorem}. Note that in Lemma~\ref{boundlemma}, we found that for each integer $N$ in the interval $[2m+3, 2m+\frac{\sqrt{16m+13}+9}{2})$, the $d$ invariants of the corresponding surgery manifold $M_{N,m}$ fails to hold the inequality in Theorem~\ref{theorem_owens_strle}. In particular the same holds if there is a squarefree integer in the interval that Lemma~\ref{boundlemma} provides. There is much research done about the existence of a squarefree integer in a given interval. In particular work of Halberstam and Roth in \cite{HB1951} ({\it cf}~ \cite[Theorem~$1.1$]{FGT2015}) guarantees that for $m$ sufficiently large there is always a squarefree integer, call it $s$, in the interval $(2m+3, 2m+\frac{\sqrt{16m+13}+9}{2})$. Combining this with Theorem~\ref{theorem_owens_strleNew}, we obtain that the desired conclusion of part $(2)$ of Theorem~\ref{main_theorem}.

\end{proof}

\bibliographystyle{amsplain}

\end{document}